\documentclass[11pt,reqno]{amsart}
\usepackage{amssymb,latexsym}
\usepackage[T2A]{fontenc}
 \usepackage[cp1251]{inputenc}
\usepackage[english]{babel}
\theoremstyle{plain}
\numberwithin{equation}{section}
\newtheorem{thm}{Theorem}
\newtheorem{lemma}{Lemma}[section]
\newtheorem{theorem}{Theorem}[section]
\newtheorem{definition}[thm]{Definition}

\begin{document}

\setcounter{page}{1}

\title[Derivations and identities ]{Derivations and identities for Kravchuk polynomials}
\author{Leonid Bedratyuk}
\address{Department of Applied Mathematics \\
                Khmelnitsky National University\\
                Khmelnitsky, Instytutska ,11\\
                29016, Ukraine}
\email{leonid.uk@gmail.com}

\begin{abstract}
We introduce the notion of Kravchuk derivations of the polynomial algebra.  We prove that any element of the kernel of the derivation gives a polynomial identity satisfied by the Kravchuk polynomials.  Also, we  prove that any kernel element of the  basic Weitzenb\"ok derivations yields a polynomial identity satisfied by the Kravchuk polynomials.  We describe the corresponding intertwining maps.

\end{abstract}

\maketitle

\begin{flushright}
 \textit{Dedicated to the 120th Anniversary of Mykhailo Kravchuk}
\end{flushright}

\section{Introduction}

The binary Kravchuk    polynomials $K_n(x,a)$  are defined by the following explicit formula
$$
K_n(x,a):=\sum_{i=0}^n (-1)^i {x \choose i} {a-x \choose n-i}, n=0,1,\ldots ,
$$
with the ordinary generating function
$$
\sum_{i=0}^{\infty} K_i(x,a) z^i=\left( 1+z \right) ^{a} \left( \frac{1-z}{1+z} \right) ^{x}.
$$
 We are interested in finding  polynomial identities satisfied by the  polynomials.  In other words, polynomials  
$P(x_0,x_1,\ldots,x_n)$ in  $n+1$ variables such that
$$
P(K_0(x,a),K_1(x,a),\ldots, K_n(x,a))=\varphi_1(a) \,\,\text{  or }  P(K_0(x,a),K_1(x,a),\ldots, K_n(x,a))=\varphi_2(x),
$$
where   $\varphi_1,$ and $ \varphi_2$ are   polynomials of one variable.
We provide a method for finding such identities which  is based on the  simple observation:
if 
$$
\frac{d}{dx} P(K_0(x,a),K_1(x,a),\ldots, K_n(x,a))=0  \text{    or } \frac{d}{da} P(K_0(x,a),K_1(x,a),\ldots, K_n(x,a))=0,
$$
then  
$
P(K_0(x,a),K_1(x,a),\ldots, K_n(x,a))$ is a function only  of $a$  or    $  P(K_0(x,a),\ldots, K_n(x,a))
$
is a function of only one  variable $x$,
i.e., it is a  polynomial identity.

On the other hand, rewrite  this derivative in the form 
\begin{gather*}
\frac{d}{dx} P(K_0(x,a),K_1(x,a),\ldots, K_n(x,a))=\\=\frac{\partial }{\partial x_0}P(x_0,x_1,\ldots,x_n)\Big |_{\{x_i=F_i(x)\}} \frac{d}{dx} K_0(x,a)+\cdots +\frac{\partial }{\partial x_n}P(x_0,\ldots,x_n)\Big |_{\{x_i=F_i(x)\}} \frac{d}{dx} K_n(x,a).
\end{gather*}
Suppose that the derivative  $\dfrac{d}{dx} K_i(x,a)$ can be expressed as a polynomial of the Kravchuk polynomials, i.e.,  
$\dfrac{d}{dx} K_i(x,a)=f_i(K_0(x,a),\ldots, K_n(x,a)),$ $ i \leq n,$
and  $f_i(x_0,x_1,\ldots,x_n) \in \mathbb{Q}[x_0,x_1,\ldots,x_n],$ $ 0 \leq i \leq n.$
Then  
 
 \begin{gather*}
\frac{d}{dx} P(K_0(x,a),K_1(x,a),\ldots, K_n(x,a))=\\=\left( \frac{\partial }{\partial x_0}P(x_0,x_1,\ldots,x_n) D(x_0) +\cdots +\frac{\partial }{\partial x_n}P(x_0,x_1,\ldots,x_n)D(x_n)\right) \Big |_{\{x_i=F_i(x)\}} =\\=
\mathcal{D}_{\mathcal{K}_1}(P(x_0,x_1,\ldots,x_n))\Big |_{\{x_i=F_i(x)\}},
\end{gather*}
 where the differential operator  $\mathcal{D}_{\mathcal{K}_1}$ on $\mathbb{Q}[x_0,x_1,\ldots,x_n]$ is defined  by 
 $
 \mathcal{D}_{\mathcal{K}_1}(x_i):=f_i(x_0,x_1,\ldots,x_n).
 $
 It is clear that if $\mathcal{D}_{\mathcal{K}_1}(P(x_0,x_1,\ldots,x_n))=0$ then $\dfrac{d}{dx} P(K_0(x,a),K_1(x,a),\ldots, K_n(x,a))=0.$
 Thus, any non-trivial polynomial  $P(x_0,x_1,\ldots,x_n)$, which belongs to the kernel of   $\mathcal{D}_{\mathcal{K}_1}$, i.e., the following holds   ${ \mathcal{D}_{\mathcal{K}_1}(P(x_0,x_1,\ldots,x_n))=0}$,   defines a polynomial identity of the form  $P(K_0(x,a),\ldots, K_n(x,a))=\varphi(a)$ for some polynomial $\varphi.$

 Analogously,\, suppose that 
  $$\dfrac{d}{da} K_i(x,a)=g_i(K_0(x,a),\ldots, K_n(x,a)), g_i \in \mathbb{Q}[x_0,x_1,\ldots,x_n],$$ and define the   differential operator $\mathcal{D}_{\mathcal{K}_2}$ by $ \mathcal{D}_{\mathcal{K}_2}(x_i):=g_i(x_0,x_1,\ldots,x_n). $
  Then any  non-trivial polynomial $P(x_0,x_1,\ldots,x_n)$, which belongs to the kernel of   $\mathcal{D}_{\mathcal{K}_2}$ defines a polynomial identity of the form  $P(K_0(x,a),\ldots, K_n(x,a))=\varphi(x)$ for some polynomial $\varphi.$
 If a polynomial $P(x_0,x_1,\ldots,x_n)$ belongs to the intersection of kernels of the operators $\mathcal{D}_{\mathcal{K}_1}$ and $\mathcal{D}_{\mathcal{K}_2}$ then it defines an identity $P(K_0(x,a),\ldots, K_n(x,a))={\rm const}.$
 
For instance, it is easily verified that
$$ 
\begin{array}{ll}
\dfrac{d}{dx} K_0(x,a)=0 & \dfrac{d}{da} K_0(x,a)=0,\\ 
&\\
\dfrac{d}{dx} K_1(x,a)=- 2K_0(x,a)&  \dfrac{d}{da} K_1(x,a)=K_0(x,a),\\
&\\
\dfrac{d}{dx} K_2(x,a)=-2K_1(x,a)& \dfrac{d}{da} K_2(x,a)=-\dfrac{1}{2}K_0(x,a)+K_1(x,a).
\end{array}
$$
Then define the differential operators $\mathcal{D}_{\mathcal{K}_1}, \mathcal{D}_{\mathcal{K}_2}$ by 
$$ 
\begin{array}{ll}
\mathcal{D}_{\mathcal{K}_1}(x_0)=0 & \mathcal{D}_{\mathcal{K}_2}(x_0)=0,\\ 
\mathcal{D}_{\mathcal{K}_1}(x_1)= x_0&  \mathcal{D}_{\mathcal{K}_2}(x_1)=x_0,\\
&\\
\mathcal{D}_{\mathcal{K}_1}(x_2)=x_1& \mathcal{D}_{\mathcal{K}_2}(x_2)=-\dfrac{1}{2}x_0+x_1.
\end{array}
$$

For the polynomial  ${x_{{1}}}^{2}-2\,x_{{2}}x_{{0}}$ we have $D_{x}(x_{{1}}^{2}-2\,x_{{2}}x_{0})=0$, thus the expression  $K_{{1}}(x,a)^{2}-2\,K_{{2}}(x,a)K_{0}(x,a)$ is a function of $a$ only. The substitution $K_n(x,a)$  of $x_n$  gives
 \begin{gather*}
K_{{1}}(x,a)^{2}-2\,K_{{2}}(x,a)K_{0}(x,a)=a.
\end{gather*}

In the same manner, the polynomial $x_{{0}}x_{{1}}-{x_{{1}}}^{2}+2\,x_{{2}}x_{{0}}$ lies in the kernel of the operator $\mathcal{D}_{\mathcal{K}_2}$. Therefore  the expression $K_{{0}}(x,a)K_{{1}}(x,a)-{K_{{1}}}(x,a)^{2}+2\,K_{{2}}(x,a)K_{{0}}(x,a)$ depends on $x$  only. In fact, after simplification we obtain 
$$
K_{{0}}(x,a)K_{{1}}(x,a)-{K_{{1}}}(x,a)^{2}+2\,K_{{2}}(x,a)K_{{0}}(x,a)=-2\,x- \left( -2\,x+a \right) ^{2}+4\,{x}^{2}-4\,ax+{a}^{2}=-2x.
$$

A similar problem  was solved  by the present author   for  the Appellpolynomials. 
 In  the paper  \cite{B_App}  was proved that any non-trivial element of kernel of the differential operator 
 $$
\mathcal{D}=x_0 \frac{\partial}{\partial x_1}+2 x_1 \frac{\partial }{\partial x_2}+\cdots+ n x_{n-1} \frac{\partial }{\partial x_n},
$$
 gives some  polynomial identity for the Appellpolynomials.
 Recall that    polynomials   $\{A_n(x) \},$  $\deg(A_n(x))=n$ is called  the Appellpolynomials if 
\begin{gather}
A'_n(x)=n A_{n-1}(x), n=0,1,2,\ldots .
\end{gather}
The  operator  $\mathcal{D}(x_i)=n x_{i-1}$  is called \textit{the basic Weitzenb\"ock derivation} and  its cernel is isomorphic to some algebra of $SL_2$-invariants.
 In the present paper we show  how the   kernel elements of the derivation $\mathcal{D}$   can be used to find a polynomials identities for the Kravchuk  polynomials.  
   A multiplicative linear map  $\psi_{AK_1}$   is called  $(\mathcal{D},\mathcal{D}_{\mathcal{K}_1})$-interwin\-ing map if the following condition holds: $\psi_{AK_1}\mathcal{D}=\mathcal{D}_{\mathcal{K}_1} \psi_{AK_1} $. Any such map induces  an isomorphism from   $\ker \mathcal{D}$  to  $\ker \mathcal{D}_{\mathcal{K}_1}.$ 
 For instance,  the discriminant of the polynomial (in the variables $X,Y$)
 
$$
x_{{0}}{X}^{3}+3\,x_{{1}}{X}^{2}Y+3\,x_{{2}}X{Y}^{2}+x_{{3}}{Y}^{3},
$$
equals
$$
\begin{vmatrix}
x_0 & 3 x_1 &3 x_2 & x_3 &0 \\
0 & x_0 & 3 x_1 &3 x_2 & x_3 \\
3x_0 & 6 x_1 &3 x_2 & 0 &0 \\
0&3x_0 & 6 x_1 &3 x_2 & 0  \\
0&0&3x_0 & 6 x_1 &3 x_2  \\
\end{vmatrix}=27 (6\,x_{{0}}x_{{3}}x_{{2}}x_{{1}}+3\,{x_{{1}
}}^{2}{x_{{2}}}^{2}-4\,{x_{{1}}}^{3}x_{{3}}-4\,{x_{{2}}}^{3}x_{{0}}-{x_{{0}}}^{2}{x_{{3}}}^{2}
),
$$
 and lies in the kernel of the operator $\mathcal{D}$. It is  well known result of the classical invariant theory. 
It is easily checked that the linear map  defined by 
\begin{align*}
&\psi_{AK_1}(x_0)=x_0,
\psi_{AK_1}(x_1)=x_1,\\
&\psi_{AK_1}(x_2)=2x_2,
\psi_{AK_1}(x_3)=-2\,x_{{1}}+6\,x_{{3}},
\end{align*}
commutes with the operators  $\mathcal{D}$ and $\mathcal{D}_{\mathcal{K}_1}$. Therefore the element
\begin{gather*}
\begin{vmatrix}
\psi_{AK_1}(x_0) & 3 \psi_{AK_1}(x_1) &3 \psi_{AK_1}(x_2) & \psi_{AK_1}(x_3) &0 \\
0 & \psi_{AK_1}(x_0) & 3\psi_{AK_1}( x_1) &3 \psi_{AK_1}(x_2) & \psi_{AK_1}(x_3) \\
3\psi_{AK_1}(x_0) & 6 \psi_{AK_1}(x_1) &3 \psi_{AK_1}(x_2) & 0 &0 \\
0&3\psi_{AK_1}(x_0) & 6 \psi_{AK_1}(x_1) &3\psi_{AK_1}( x_2) & 0  \\
0&0&3\psi_{AK_1}(x_0) & 6 \psi_{AK_1}(x_1) &3 \psi_{AK_1}(x_2)  \\
\end{vmatrix}=\\=\begin{vmatrix}
x_0 & 3 x_1 &6 x_2 & -2\,x_{{1}}+6\,x_{{3}} &0 \\
0 & x_0 & 3 x_1 &6 x_2 & -2\,x_{{1}}+6\,x_{{3}} \\
3x_0 & 6 x_1 &6 x_2 & 0 &0 \\
0&3x_0 & 6 x_1 &6 x_2 & 0  \\
0&0&3x_0 & 6 x_1 &6 x_2  
\end{vmatrix},
\end{gather*}
lies in the kernel of the operator  $\mathcal{D}_{\mathcal{K}_1}$ and defines the following identity for the Kravchuk polynomials:

\begin{gather*}
\begin{vmatrix}
K_0(x,a) & 3 K_1(x,a) &6 K_2(x,a) & -2\,K_{{1}}(x,a)+6\,K_{{3}}(x,a) &0 \\
0 & K_0(x,a) & 3 K_1(x,a) &6 K_2(x,a) & -2\,K_{{1}}(x,a)+6\,K_{{3}}(x,a) \\
3K_0(x,a) & 6 K_1(x,a) &6 K_2(x,a) & 0 &0 \\
0&3K_0(x,a) & 6 K_1(x,a) &6 K_2(x,a) & 0  \\
0&0&3K_0(x,a) & 6 K_1(x,a) &6 K_2(x,a)\end{vmatrix}=\\ \\=108a^3.
\end{gather*}

In the  paper we present methods of the theory of locally  nilpotent  derivation  to find  polynomial  identities for  the Kravchuk polynomials.

In section 2,we give  a brief introduction to the theory of locally nilpotent derivations. Also, we  introduce the notion of the Kravchuk derivations and find its kernels. In  this way we obtain  some  polynomials identities for the Kravchuk polynomials.

In the section 3 we find a $(\mathcal{D},\mathcal{D}_{\mathcal{K}_1})$-intertwining map and a $(\mathcal{D},\mathcal{D}_{\mathcal{K}_2})$-intertwining map.


\section{Kravchuk  derivations }

\subsection{Derivations and its kernels} Let   $\mathbb{Q}[x_0,x_1,x_2,\ldots,x_n]$ be the polynomial algebra  in  $n+1$ variables $x_0,x_1,x_2,\ldots,x_n$ over $\mathbb{Q}.$ Recall  that a  {\it derivation} of the polynomial algebra $\mathbb{Q}[x_0,x_1,x_2,\ldots,x_n]$ is a linear  map  $D$ satisfying the Leibniz rule: 
$$
D(x_1 \, x_2)=D(x_1) x_2+x_1 D(x_2), \text{  for all }  x_1, x_2 \in \mathbb{Q}[x_0,x_1,x_2,\ldots,x_n].
$$
A derivation $D$  is called {\it locally nilpotent} if for every $f \in \mathbb{Q}[x_0,x_1,x_2,\ldots,x_n]$ there is an $n \in \mathbb{N}$ such that $D^n(f)=0.$ 
Any derivation   $D$ is completely determined by the elements $D(x_i).$ A  derivation   $D$  is called  \textit{linear} if  $D(x_i)$ is a linear form. A  linear locally nilpotent derivation is called a \textit{Weitzenb\"ock derivation}. 
A derivation    $D$ 
is called the triangular if  $D(x_i) \in \mathbb{Q}[x_0,\ldots,x_{i-1}].$ Any triangular derivation is locally nilpotent.

The subalgebra 
$$
\ker D:=\left \{ f \in \mathbb{Q}[x_0,x_1,x_2,\ldots,x_n] \mid  D(f)=0 \right \},
$$
is called the {\it kernel} of derivation $D.$

For arbitrary locally nilpotent derivation  $D$ the following statement holds:
\begin{theorem} \label{maint} Suppose that there exists  a polynomials $h$ such that  $D(h) \neq 0$ but $D^2(h)=0.$ Then
$$
\ker D=\mathbb{Q}[\sigma(x_0),\sigma(x_1),\ldots,\sigma(x_n)][D(h)^{-1}] \cap \mathbb{Q}[x_0,x_1,\ldots,x_n],
$$
where $\sigma$ is the Diximier map 
$$
\sigma(x_i)=\sum_{k=0}^{\infty} D^k(x_i) \frac{\lambda^k}{k!},\lambda=-\dfrac{h}{D(h)}, D(\lambda)=-1.
$$
\end{theorem}
The element $\lambda$ is called a slice of the locally nilpotent derivation $D.$
The  proof one may find   in \cite{Now} and \cite{Essen}.


\subsection{Derivatives of the Kravchuk polynomials}

The  ordinary generating function for the Kravchuk polynomials has the form 
$$
\sum_{i=0}^{\infty} K_i(x,a) z^i=\left( 1+z \right) ^{a} \left( \frac{1-z}{1+z} \right) ^{x}.
$$

Differentiating  with respect to $x$ we get the generation function for the derivative  $\dfrac{d}{dx} K_i(x,a):$
$$
\sum_{i=0}^{\infty} \frac{d}{dx} K_i(x,a) z^i=\left( 1+z \right) ^{a} \left( \frac{1-z}{1+z} \right) ^{x} \ln\left(\frac{1-z}{1+z} \right)=\left(\sum_{i=0}^{\infty} K_i(x,a) z^i \right) \ln\left(\frac{1-z}{1+z} \right).
$$
Taking into account  
\begin{gather*}
\ln\left(\frac{1-z}{1+z} \right)=\ln(1-z)-\ln(1+z)=-\sum_{i=1}^{\infty}\frac{z^i}{i}+\sum_{i=1}^{\infty}(-1)^i\frac{z^i}{i}= \sum_{i=1}^{\infty} \left(-\frac{2}{2i-1}\right)z^{2i-1}=\\=-2 \sum_{i=1}^{\infty}\frac{1-(-1)^i}{2i}z^i,
\end{gather*}
we  have 
\begin{gather*}
\frac{d}{dx} K_{n}(x,a)=[z^n]\left(\sum_{i=0}^{\infty} K_i(x,a) z^i \right) \left(-2 \sum_{i=1}^{\infty}\frac{1-(-1)^i}{2i}z^i \right)=\\=-2 \sum_{i=1}^n \frac{1-(-1)^i}{2i}K_{n-i}(x,a).
\end{gather*}
In  the paper  \cite{Kra} one may see another proof of the formula and in  \cite{Koep} is presented another expression for $\dfrac{d}{dx} K_{n}(x,a)$.

Differentiating the generating function  with respect to $a$ we get
\begin{gather*}
\sum_{i=0}^{\infty} \frac{d}{da} K_i(x,a) z^i=\left( 1+z \right) ^{a} \left( \frac{1-z}{1+z} \right) ^{x} \ln(1+z)=\left( \sum_{i=1}^{\infty}  K_i(x,a) z^i \right) \left( \sum_{i=1}^{\infty}(-1)^{i+1}\frac{z^i}{i}\right).
\end{gather*}
In  the same way,   by using that  
$$
\sum_{i=0}^{\infty} \frac{d}{da} K_i(x,a) z^i=\left(\sum_{i=0}^{\infty} K_i(x,a) z^i \right) \ln\left(1+z \right),
$$
 we obtain an  expression for the derivatives with respect to  $a:$
$$
\frac{d}{da} K_{n}(x,a)=\sum_{i=0}^{n-1} \frac{(-1)^{n+1+i}}{n-i}K_{i}(x,a).
$$

\subsection{Kravchuk derivations}
The expressions  for $\dfrac{d}{dx} K_{n}(x,a)$ and $\dfrac{d}{da} K_{n}(x,a)$ motivate the following definition  
\begin{definition}
Derivations of  $\mathbb{Q}[x_0,x_1,x_2,\ldots,x_n]$  defined by 
 \begin{align*}
&D_{\mathcal{K}_1}(x_0)=0, D_\mathcal{K}(x_n)= \sum_{i=1}^{n} \frac{1-(-1)^{i}}{2i}\, x_{n-i},\\
& D_{\mathcal{K}_2}(x_0)=0, D_{\mathcal{K}_2}(x_n)= \sum_{i=0}^{n-1} \frac{(-1)^{n+1+i}}{n-i}\, x_{i}, n=1,2,\ldots, n, \ldots.,
\end{align*}
are called  {\bf the first Kravchuk derivation}  and  {\bf the second  Kravchuk derivation}  respectively.
\end{definition}
We  have
$$
\begin{array}{ll}
D_{\mathcal{K}_1}(x_0)=0,& D_{\mathcal{K}_2}(x_0)=0,\\
D_{\mathcal{K}_1}(x_1)=x_0,&D_{\mathcal{K}_2}(x_1)=x_0,\\
 D_{\mathcal{K}_1}(x_2)=x_1,& D_{\mathcal{K}_2}(x_2)=-\dfrac12\,x_{{0}}+x_{{1}},\\
 & \\
 D_{\mathcal{K}_1}(x_3)=\dfrac13\,x_{{0}}+x_{{2}},& D_{\mathcal{K}_2}(x_3)=\dfrac 13\,x_{{0}}-\dfrac 12\,x_{{1}}+x_{{2}},\\
  & \\
 D_{\mathcal{K}_1}(x_4)=\dfrac13\,x_{{1}}+x_{{3}},& D_{\mathcal{K}_2}(x_4)=-\dfrac14\,x_{{0}}+\dfrac13\,x_{{1}}-\dfrac12\,x_{{2}}+x_{{3}},\\
  & \\
 D_{\mathcal{K}_1}(x_5)=\dfrac 15\,x_{{0}}+\dfrac13\,x_{{2}}+x_{{4}},& D_{\mathcal{K}_2}(x_5)=\dfrac15\,x_{{0}}-\dfrac14\,x_{{1}}+\dfrac13\,x_{{2}}-\dfrac12\,x_{{3}}+x_{{4}},\\
  & \\
 D_{\mathcal{K}_1}(x_6)=\dfrac15\,x_{{1}}+\dfrac 13\,x_{{3}}+x_{{5}}, & D_{\mathcal{K}_2}(x_6)=-\dfrac16\,x_{{0}}+\dfrac15\,x_{{1}}-\dfrac14\,x_{{2}}+\dfrac13\,x_{{3}}-\dfrac12\,x_{{4}}+x_{{
5}}
.
\end{array}
$$

We define  the substitution homomorphism
 $\varphi_{\mathcal{K}}:\mathbb{Q}[x_0,x_1,\ldots,x_n] \to \mathbb{Q}[x]  $  by   $\varphi_{\mathcal{K}}(x_i)=K_i(x,a)).$ 
Put
\begin{align*}
&\ker\varphi_{\mathcal{K}}:=\{ P ( x_0,x_1,...,x_n) \in  \mathbb{Q}[ x_0,x_1,...,x_n] \mid \varphi_{\mathcal{K}}(P( x_0,x_1,...,x_n)) \in \mathbb{Q} \},\\
&\ker\varphi_{\mathcal{K}_1}:=\{ P ( x_0,x_1,...,x_n) \in  \mathbb{Q}[ x_0,x_1,...,x_n] \mid \varphi_{\mathcal{K}}(P( x_0,x_1,...,x_n)) \in \mathbb{Q}[a] \},\\
&\ker\varphi_{\mathcal{K}_2}:=\{ P ( x_0,x_1,...,x_n) \in  \mathbb{Q}[ x_0,x_1,...,x_n] \mid \varphi_{\mathcal{K}}(P( x_0,x_1,...,x_n)) \in \mathbb{Q}[x] \}.
\end{align*}
Any element of $\ker \varphi_{\mathcal{K}_1}$  and $\ker \varphi_{\mathcal{K}_2}$  gives a polynomial identity for the Kravchuk polynomials respectively. Put 
\begin{gather*}
\ker \mathcal{D}_{\mathcal{K}_1}:=\{ S \in  \mathbb{Q}[ x_0,x_1,...,x_n]  \mid \mathcal{D}_{\mathcal{K}_1}(S) =0 \},\\
\ker \mathcal{D}_{\mathcal{K}_2}:=\{ S \in  \mathbb{Q}[ x_0,x_1,...,x_n]  \mid \mathcal{D}_{\mathcal{K}_2}(S) =0 \} .
\end{gather*}
It is  easy to see that   $\varphi_{\mathcal{K}}  \mathcal{D}_{\mathcal{K}_1} =\dfrac{d}{dx}\, \varphi_{\mathcal{K}}$  and $\varphi_{\mathcal{K}}  \mathcal{D}_{\mathcal{K}_2} =\dfrac{d}{da}\, \varphi_{\mathcal{K}}$.  It follows that $\varphi_{\mathcal{K}} (\ker \mathcal{D}_{\mathcal{K}_1})\subset \ker_1 \varphi_{\mathcal{K}}$  and $\varphi_{\mathcal{K}} (\ker \mathcal{D}_{\mathcal{K}_2})\subset \ker_2 \varphi_{\mathcal{K}}$.

We have thus proved the following theorem. 
\begin{theorem}\label{M-T}  Let  $P(x_0,x_1,\ldots,x_n)$ be a polynomial. 

(i) If $\mathcal{D}_{\mathcal{K}_1}(P(x_0,x_1,\ldots,x_n))=0$ then  $P(K_0(x,a),K_1(x,a),\ldots, K_n(x,a))\in \mathbb{Q}[a]$;

(ii) if  $\mathcal{D}_{\mathcal{K}_2}(P(x_0,x_1,\ldots,x_n))=0$ then $P(K_0(x,a),K_1(x,a),\ldots, K_n(x,a)) \in \mathbb{Q}[x]$.
\end{theorem}

Note that $\varphi_{\mathcal{K}} (\ker \mathcal{D}_{\mathcal{K}_1}) \neq  \ker \varphi_{\mathcal{K}}$ and $\varphi_{\mathcal{K}} (\ker \mathcal{D}_{\mathcal{K}_2}) \neq  \ker \varphi_{\mathcal{K}}$. In fact, we  have  $$\varphi_{\mathcal{K}}(x_{{3}}{x_{{1}}}^{2}-2\,x_{{2}}x_{{3}}x_{{0}}-x_{{1}}x_{{4}}x_{{0}}-3
\,x_{{3}}{x_{{0}}}^{2}+5\,x_{{5}}{x_{{0}}}^{2}
)=0$$ but $x_{{3}}{x_{{1}}}^{2}-2\,x_{{2}}x_{{3}}x_{{0}}-x_{{1}}x_{{4}}x_{{0}}-3
\,x_{{3}}{x_{{0}}}^{2}+5\,x_{{5}}{x_{{0}}}^{2} \notin \ker \mathcal{D}_{\mathcal{K}_{1,2}}.$


\subsection{The kernel of the first Kravchuk  derivation}

It is obviously that that Kravchuk  derivations are triangular and thus   locally nilpotent. Thus to find its kernels we may use  the Theorem  \ref{maint}.

Let us construct  the Diximier map for the first Kravchuk  derivation. For this purpose, we derive first a close expression for  the powers $D^k_{\mathcal{K}_1}(x_n).$ 
 Differentiating $k$ times the generating function  with respect to $x$ we get
$$
\sum_{i=1}^{\infty}  K_i(x,a) z^i= \ln\left(\frac{1-z}{1+z} \right),
$$
and taking into account 
$$
\sum_{i=1}^{\infty} \frac{d}{dx} K_i(x,a) z^i=\left(\sum_{i=0}^{\infty} K_i(x,a) z^i \right) \ln\left(\frac{1-z}{1+z} \right),
$$
we get 
\begin{gather*}
\sum_{i=k}^{\infty}\frac{d^k}{dx^k} K_i(x,a) z^i=\left( \sum_{i=0}^{\infty} K_i(x,a)z^i \right) \left(\ln  \left( {\frac {1+z}{1-z}} \right)\right)^k.
\end{gather*}

By using the expansion  
\begin{gather*}
\left(\ln\left( {\frac {1+z}{1-z}} \right) \right)^k=
\left( \sum_{i=k}^{\infty} S^{(k)}(i) z^i \right),
\end{gather*}
where 
$$
S^{(k)}(n)=\sum _{m=k}^{n}{n-1\choose m-1}{\dfrac {{2}^{m} k!}{m!}}{\it s} \left( m,k \right),
$$
and  $s(m,k)$ are the Stirling numbers of the first kind.
Then 
\begin{gather*}
\sum_{i=k}^{\infty}\frac{d^k}{dx^k} K_i(x,a) z^i=\sum_{n=k}^{\infty}\left( \sum_{i=0}^{n-k}   K_i(x,a) S^{(k)}(n-i) \right) z^n.
\end{gather*}
Thus
\begin{gather*}
\mathcal{D}^k_{\mathcal{K}_1}(x_n)=\sum_{i=0}^{n-k}   x_i S^{(k)}(n-i).
\end{gather*}
 Now we may find the Diximier map:
\begin{gather*}
\sigma(x_n)=\sum_{k=0}^{n}D^k_{\mathcal{K}_1}(x_n) \frac{\lambda^k}{k!}
=\sum_{k=0}^{n} \frac{\lambda^k}{k!} \sum_{i=0}^{n-k}   x_i S^{(k)}(n-i)=
\sum_{i=0}^{n} x_i \sum_{k=0}^{n-i} \frac{\lambda^k}{k!} S^{(k)}(n-i).
\end{gather*}
Replacing  $\lambda$ by $-\displaystyle  \frac{x_1}{x_0} $, we obtain, after   simplifying:
\begin{gather*}
\sigma(x_n)=  \sum_{i=0}^{n} x_i \sum_{k=0}^{n-i} \frac{(-1)^k}{k!} \left(\frac{x_1}{x_0}\right)^k S^{(k)}(n-i)=\\=x_0 \frac{(-1)^n}{n!}\left(\frac{x_1}{x_0}\right)^n+x_1 \frac{(-1)^{n-1}}{(n-1)!}\left(\frac{x_1}{x_0}\right)^{n-1}+ \sum_{i=0}^{1} x_i \sum_{k=0}^{n-i-1} \frac{(-1)^k}{k!} \left(\frac{x_1}{x_0}\right)^k S^{(k)}(n-i)+\\+\sum_{i=2}^{n} x_i \sum_{k=0}^{n-i} \frac{(-1)^k}{k!} \left(\frac{x_1}{x_0}\right)^k S^{(k)}(n-i)=
\frac{(-1)^{n-1}}{n (n-2)!} \frac{x_1^n}{x_0^{n-1}}+ \sum_{i=0}^{1} x_i \sum_{k=0}^{n-i-1} \frac{(-1)^k}{k!} \left(\frac{x_1}{x_0}\right)^k S^{(k)}(n-i)+\\ +\sum_{i=2}^{n} x_i \sum_{k=0}^{n-i} \frac{(-1)^k}{k!} \left(\frac{x_1}{x_0}\right)^k S^{(k)}(n-i).
\end{gather*}
The polynomials
\begin{multline*}
C_n:= n(n-2)! x_0^{n-1}\sigma(x_n)=(-1)^{n-1}x_1^n+n(n-2)!\sum_{i=0}^{1} x_i \sum_{k=0}^{n-i-1} \frac{(-1)^k}{k!} x_0^{n-1-k} x_1^k S^{(k)}(n-i)+\\+n(n-2)!\sum_{i=2}^{n} x_i \sum_{k=0}^{n-i} \frac{(-1)^k}{k!} x_0^{n-1-k} x_1^k S^{(k)}(n-i), n>1,
\end{multline*}
belong to the kernel $\ker \mathcal{D}_{\mathcal{K}_1}.$ We call them   \textit{the Cayley elements} of the locally nilpotent derivation $\mathcal{D}_{\mathcal{K}_1}$.  
The first few Cayley elements are shown below:
\begin{align*}
&C_2=2x_2x_0-x_1^2,\\
&C_3=3\,x_{{3}}{x_{{0}}}^{2}-x_{{1}}{x_{{0}}}^{2}-3\,x_{{1}}x_{{0}}x_{{2}}+
{x_{{1}}}^{3}
,\\
&C_4=8\,x_{{4}}{x_{{0}}}^{3}-8\,x_{{1}}{x_{{0}}}^{2}x_{{3}}+4\,{x_{{1}}}^{2
}x_{{0}}x_{{2}}-{x_{{1}}}^{4}
,\\
&C_5=30\,x_{{5}}{x_{{0}}}^{4}-30\,x_{{1}}{x_{{0}}}^{3}x_{{4}}-10\,x_{{1}}{x
_{{0}}}^{3}x_{{2}}-6\,x_{{1}}{x_{{0}}}^{4}+5\,{x_{{1}}}^{3}{x_{{0}}}^{
2}+15\,{x_{{1}}}^{2}{x_{{0}}}^{2}x_{{3}}-5\,{x_{{1}}}^{3}x_{{0}}x_{{2}
}+{x_{{1}}}^{5}
,\\
&C_6=144\,x_{{6}}{x_{{0}}}^{5}+8\,{x_{{1}}}^{2}{x_{{0}}}^{4}-48\,x_{{1}}{x_
{{0}}}^{4}x_{{3}}-144\,x_{{1}}{x_{{0}}}^{4}x_{{5}}+48\,{x_{{1}}}^{2}{x
_{{0}}}^{3}x_{{2}}+72\,{x_{{1}}}^{2}{x_{{0}}}^{3}x_{{4}}-16\,{x_{{1}}}
^{4}{x_{{0}}}^{2}-\\&-24\,{x_{{1}}}^{3}{x_{{0}}}^{2}x_{{3}}+6\,{x_{{1}}}^{
4}x_{{0}}x_{{2}}-{x_{{1}}}^{6}
.
\end{align*} 

Theorem \ref{maint} implies

\begin{theorem}
\begin{align*}
&\ker {\mathcal{D}_{\mathcal{K}_1}}=\mathbb{Q}[x_0,x_1,C_3,C_4,\ldots,C_n][x_1^{-1}] \cap \mathbb{Q}[x_0,x_1,\ldots,x_n].
\end{align*}
\end{theorem}

Thus, we obtain  a description of the kernel of the Kravchuk  derivation.

To get an identity for the Kravchuk polynomials we should find  $\varphi_F(C_n).$

 We  have
\begin{gather*}
\varphi_{\mathcal{K}}(C_n)=\varphi_K(\sigma(x_n))=
\sum_{i=0}^{n} K_i \sum_{k=0}^{n-i} \frac{(-1)^k}{k!}  K_1^k S^{(k)}(n-i).
\end{gather*}
By Theorem   \ref{M-T} the right side is a function of  $a$. We have 
\begin{align*}
&\varphi_{\mathcal{K}}(C_2)=-a\\
&\varphi_{\mathcal{K}}(C_3)=0\\
&\varphi_{\mathcal{K}}(C_4)=a \left( a-2 \right) \\
&\varphi_{\mathcal{K}}(C_5)=0\\
&\varphi_{\mathcal{K}}(C_6)=-3\,a \left( a-2 \right)  \left( a-4 \right) \\
\end{align*}

We would like to propose the following conjecture  for the  Kravchuk polynomials:

\noindent
\textbf{Conjecture 1.} 
\begin{gather*}
\sum_{i=0}^{n} K_i(x,a) \sum_{k=0}^{n-i} \frac{(-1)^k}{k!}  K_1(x,a)^k S^{(k)}(n-i)=\\=\left\{ \begin{array}{ll} 0, &  n  \text{ odd },\\ (-1)^m(2m-1)!!\, a (a-2)(a-4) \ldots (a-2(m-1)), & n=2m. \end{array} \right.,
\end{gather*}
where 
$$
S^{(k)}(n)=\sum _{m=k}^{n}{n-1\choose m-1}{\dfrac {{2}^{m} k!}{m!}}{\it s} \left( m,k \right).
$$

\subsection{The kernel of the second Kravchuk  derivation}

In the same way, to   derive  a close expression for  the powers $D^k_{\mathcal{K}_2}(x_n)$
we use the exponential generating function  for the Stirling numbers of the first kind $s(n,k):$
$$
\sum_{n=k}^{\infty}s(n,k) \frac{z^{i}}{n!}= \frac{\left(\ln(1+z)\right)^k}{k!}.
$$
We have 
\begin{gather*}
\sum_{i=k}^{\infty}\frac{d^k}{da^k} K_i(x,a) z^i=\left( \sum_{i=0}^{\infty} K_i(x,a)z^i \right) \left(\ln  \left(1+z\right)\right)^k=\\=\left( \sum_{i=0}^{\infty} K_i(x,a)z^i \right)  \left( \sum_{i=k}^{\infty} \frac{k!}{n!}s(n,k) z^i \right)=\sum_{n=k}^{\infty}\left( \sum_{i=0}^{n-k}   K_i(x,a) \frac{k!}{(n-i)!}s( n-i,k)\right)z^n.
\end{gather*}
Thus
\begin{gather*}
\mathcal{D}^k_{\mathcal{K}_2}(x_n)=\sum_{i=0}^{n-k}   x_i \frac{k!}{(n-i)!} s(n,k).
\end{gather*}
Since  $D_{\mathcal{K}_2}\left(-\displaystyle  \frac{x_1}{x_0} \right)=-1$  we put $\lambda=-\displaystyle  \frac{x_1}{x_0}.$ Now we may find the Diximier map:
\begin{gather*}
\sigma(x_n)=\sum_{k=0}^{n}D^k_{\mathcal{K}_2}(x_n) \frac{\lambda^k}{k!}
=\\=\sum_{k=0}^{n} \frac{\lambda^k}{k!} \sum_{i=0}^{n-k}   x_i \frac{k!}{(n-i)!}\, s( n-i,k )=
\sum_{i=0}^{n}  x_i \sum_{k=0}^{n-i} \frac{\lambda^k}{(n-i)!} \,s( n-i,k ).
\end{gather*}
Replacing  $\lambda$ by $-\displaystyle  \frac{x_1}{x_0} $, we obtain:
\begin{gather*}
\sigma(x_n)= \sum_{i=0}^{n}  x_i \sum_{k=0}^{n-i}  \frac{(-1)^k}{(n-i)!} \frac{x_1^k}{x_0^k} \,s(n-i,k).
\end{gather*}

The first few Cayley elements are shown below:
\begin{align*}
&\sigma(x_2)=\frac12\,{\frac {x_{{1}}x_{{0}}-{x_{{1}}}^{2}+2\,x_{{2}}x_{{0}}}{x_{{0}}}},\\
&\sigma(x_3)=-\frac13\,{\frac {x_{{1}}{x_{{0}}}^{2}-{x_{{1}}}^{3}+3\,x_{{2}}x_{{1}}x_{{0
}}-3\,x_{{3}}{x_{{0}}}^{2}}{{x_{{0}}}^{2}}}
,\\
&\sigma(x_4)=\frac18\,{\frac {2\,x_{{1}}{x_{{0}}}^{3}+{x_{{1}}}^{2}{x_{{0}}}^{2}-2\,{x_
{{1}}}^{3}x_{{0}}-{x_{{1}}}^{4}+4\,x_{{2}}x_{{1}}{x_{{0}}}^{2}+4\,x_{{
2}}{x_{{1}}}^{2}x_{{0}}-8\,x_{{3}}x_{{1}}{x_{{0}}}^{2}+8\,x_{{4}}{x_{{0
}}}^{3}}{{x_{{0}}}^{3}}}
.
\end{align*} 
By using the substantial homomorphism  $\varphi_{\mathcal{K}}$ we get:
\begin{align*}
&\varphi_{\mathcal{K}}(\sigma(x_2))=\frac12\,( {K_{{1}}(x,a)K_{{0}}(x,a)-{K_{{1}}}(x,a)^{2}+2\,K_{{2}}(x,a)K_{{0}}}(x,a))=-x,\\
&\varphi_{\mathcal{K}}(\sigma(x_3))=-\frac13\,(K_{{1}}(x,a){K_{{0}}}(x,a)^{2}-{K_{{1}}}(x,a)^{3}+3\,K_{{2}}(x,a)K_{{1}}(x,a)K_{{0}}(x,a)-\\&-3\,K_{{3}}(x,a){K_{{0}}}(x,a)^{2})=0,\\
&\varphi_{\mathcal{K}}(\sigma(x_4))=\frac18\,(2\,K_{{1}}(x,a){K_{{0}}}(x,a)^{3}+{K_{{1}}}(x,a)^{2}{K_{{0}}}(x,a)^{2}-2\,{K_{{1}}}(x,a)^{3}K_{{0}}(x,a)-\\&{K_{{1}}}(x,a)^{4}+4\,K_{{2}}(x,a)K_{{1}}(x,a){K_{{0}}}(x,a)^{2}+4\,K_{{2}}(x,a){K_{{1}}}(x,a)^
{2}K_{{0}}(x,a)-\\&-8\,K_{{3}}(x,a)K_{{1}}(x,a){K_{{0}}}(x,a)^{2}+8\,K_{{4}}(x,a){K_{{0}}}(x,a)^{3})
=\frac12\, \left( x-1 \right) x,\\
&\varphi_{\mathcal{K}}(\sigma(x_5))=0,\\
&\varphi_{\mathcal{K}}(\sigma(x_6))=-\frac16\,x \left( x-1 \right)  \left( x-2 \right), \\
\end{align*}

For the general case we propose the following conjecture:

\noindent
\textbf{Conjecture 2.} 
\begin{gather*}
\sum_{i=0}^{n}  K_i(x,a) \sum_{k=0}^{n-i}  \frac{(-1)^k}{(n-i)!}K_1(x,a)^k s(n-i, k) =\left\{ \begin{array}{ll} 0, &  n  \text{ odd},\\ \displaystyle  (-1)^{m} { x \choose m} , & n=2\, m . \end{array} \right.
\end{gather*}


\section{Appel-Kravchuk    intertwining maps}


 \subsection{The first Kravchuk derivation}

Denote       by  $\psi_{AK_1}$  an  $(\mathcal{D},\mathcal{D}_{\mathcal{K}_1})$  intertwining map. Suppose it has the form :
$$
\psi_{AK_1}(x_0)=x_0, \psi_{AK_1}(x_n)=\sum_{i=1}^n T(n,i)x_i.
$$
After direct  calculations we obtain 
\begin{align*}
&\psi_{AK_1}(x_0)=x_0,
\psi_{AK_1}(x_1)=x_1,
\psi_{AK_1}(x_2)=2x_2,\\
&\psi_{AK_1}(x_3)=-2\,x_{{1}}+6\,x_{{3}},
\psi_{AK_1}(x_4)=-16\,x_{{2}}+24\,x_{{4}},\\
&\psi_{AK_1}(x_5)=16\,x_{{1}}-120\,x_{{3}}+120\,x_{{5}},
\psi_{AK_1}(x_6)=272\,x_{{2}}-960\,x_{{4}}+720\,x_{{6}}.
\end{align*}

Let us prove the following statement
\begin{theorem} The numbers $T(n,i)$ are the following explicit form
$$
 T(n,i)=\displaystyle \sum_{j=i}^n (-1)^{j-i} 2^{n-j} j! S(n,j)  {j-1 \choose i-1},
$$
where  $S(n,j)$ are  the  Stirling numbers of the second kind. 
\end{theorem}
\begin{proof}
To make sure that the initial conditions hold  we calculate  $T(0,0)=1, T(1,1)=0, T(2,1)=0,T(2,2)=2$. 
By using the summation formula
$$
\sum_{i=1}^n a_i \sum_{j=1}^i b_j c_{i-j}=\sum_{i=0}^{n-1}c_i \sum_{j=i+1}^n a_j b_{j-i},
$$
we have 
\begin{gather*}
\mathcal{D}_{\mathcal{K}_1}\psi_{AK_1}(x_n))=\sum_{i=1}^n T(n,i) \mathcal{D}_{\mathcal{K}_1} (x_i)= \sum_{i=1}^n T(n,i) \sum_{j=1}^{i} \frac{1-(-1)^{j}}{2j}\, x_{i-j}=\\=\sum_{i=0}^{n-1}x_i \sum_{j=i+1}^n \frac{1-(-1)^{j-i}}{2(j-i)}T(n,j).
\end{gather*}
On the other hand
$$
D_{\mathcal{K}_1}(\psi_{AK_1}(x_n))=\psi_{AK_1}(\mathcal{D}(x_n))=n \psi_{AK_1}(x_{n-1})=n \sum_{i=1}^{n-1} T(n-1,i) x_{i}.
$$
By equation the corresponding coefficients we obtain  the recurrence relations for  the sequences  $T(n,i):$
\begin{align*}
&\sum_{j=i+1}^n \frac{1-(-1)^{j-i}}{2(j-i)}T(n,j)= n T(n-1,i),i=0,1,\ldots, n-1\\
\end{align*}
In order to prove the theorem, we need the following important lemma
\begin{lemma}
Let us consider  two sequences   $a_n$ and  $A(n,k),$ $A(n,k)=0, k>n$  with the generating functions:
$$
 \sum_{n=k}^{\infty}A(n,k)\frac{z^n}{n!}=(f(z))^k,
 \sum_{n=0}^{\infty} a_n z^k=g(z).$$
Suppose  that  $g(f(z))=z.$

Then  $$\sum_{j=i+1}^n a_{j-i} A(n,j)=\sum_{k=1}^{n-i}a_n A(n,k+i)=n A(n-1,i),i=0,1,\ldots, n-1.$$
\end{lemma}
\begin{proof}
Multiply the left  side  by $\dfrac{z^n}{n!}$ and sum on $n$  from $i$   to $\infty$ we get:
\begin{gather*}
\sum_{n=i}^{\infty}\sum_{j=i+1}^n a_{j-i} A(n,j)\frac{z^n}{n!}=\sum_{j=i+1}^\infty a_{j-i} \sum_{n=i}^{\infty}A(n,j)z^n=\sum_{j=i+1}^\infty a_{j-i} (f(z))^j=\sum_{k=1}^\infty a_k (f(z))^{i+k}=\\=(f(z))^i \sum_{k=1} a_k (f(z))^{k}=(f(z))^i g(f(z))=z (f(z))^i=z \sum_{n=0}^\infty A(n,i) \frac{z^n}{n!}=\sum_{n=1}^\infty n A(n-1,i) \frac{z^n}{n!}.
\end{gather*}
Since 
$$
\sum_{n=0}^{\infty}\sum_{j=i+1}^n a_{j-i} A(n,j)\frac{z^n}{n!}=\sum_{n=0}^{\infty} \bigl(\sum_{j=i+1}^n a_{j-i} A(n,j)  \bigr) \frac{z^n}{n!},
$$
then by equation the corresponding coefficients we get the identities
$$
\sum_{j=i+1}^n a_{j-i} A(n,j)=\sum_{k=1}^{n-i}a_n A(n,k+i)=n A(n-1,i),
$$
as required.
\end{proof}
Now, let us find the exponential generating function for  the numbers 
$$
T(n,i)=\displaystyle \sum_{j=i}^n (-1)^{j-i} 2^{n-j} j! S(n,j)  {j-1 \choose i-1}.
$$
Multiply the left and right sides  by $\dfrac{z^n}{n!}$ and sum on $n$  from $i$   to $\infty.$ Taking into account  that 
$$
\sum_{n=i}^{\infty} S(n,j)\frac{(z)^n}{n!}=\frac{(e^{z}-1)^j}{j!}
$$
 we obtain

\begin{gather*}
\sum_{n=i}^{\infty}T(n,i)\frac{z^n}{n!}=\sum_{n=i}^{\infty} \sum_{j=i}^\infty (-1)^{j-i} 2^{n-j} j!{j-1 \choose i-1} S(n,j)\frac{z^n}{n!} = \\=\sum_{j=i}^\infty \frac{(-1)^{j-i}}{ 2^{j}} j!{j-1 \choose i-1} \sum_{n=i}^{\infty} S(n,j)\frac{(2z)^n}{n!}=\sum_{j=i}^\infty \frac{(-1)^{j-i}}{ 2^{j}} {j-1 \choose i-1} (e^{2z}-1)^j.
\end{gather*}

By using the expansion
$$
\sum_{j=i}^\infty \frac{(-1)^{j-i}}{ 2^{j}} {j-1 \choose i-1} z^j=\left( \frac{z}{z+2}\right)^i,
$$
we get
$$
\sum_{n=i}^{\infty}T(n,i)\frac{z^n}{n!}=\left( \frac{e^{2z}-1}{e^{2z}+1} \right)^i.
$$

Consider the sequence 
$$
a_n =\frac{1-(-1)^n}{2n}.
$$
Its ordinary generating function has the form 
$$
g(z)= \sum_{n=1}^{\infty} a_n z^k=\frac 12 \ln\left( \frac{1+z}{1-z} \right).
$$
Put  $f(z)=\dfrac{e^{2z}-1}{e^{2z}+1}$. It is easy to see that  
 $g(f(z))=z$. Therefore, by Lemma 3.1, we get that the numbers $T(n,i)$ are solutions of the system of recurrence equations
$$
\sum_{j=i+1}^n \frac{1-(-1)^{j-i}}{2(j-i)}T(n,j)= n T(n-1,i),i=0,1,\ldots, n-1,
$$
which is what had to be proved.
\end{proof}



 \subsection{The second Kravchuk derivation}


Denote       by  $\psi_{AK_2}$  an  $(\mathcal{D},\mathcal{D}_{\mathcal{K}_2})$  intertwining map. Suppose it has the form :
$$
\psi_{AK_2}(x_0)=x_0, \psi_{AK_2}(x_n)=\sum_{i=1}^n B(n,i)x_i, n>0
$$
We have 
\begin{align*}
&\psi_{AK_2}(x_0)=x_0,
\psi_{AK_2}(x_1)=x_1,
\psi_{AK_2}(x_2)=x_{{1}}+2\,x_{{2}},\\
&\psi_{AK_2}(x_3)=x_{{1}}+6\,x_{{2}}+6\,x_{{3}},
\psi_{AK_2}(x_4)=x_{{1}}+14\,x_{{2}}+36\,x_{{3}}+24\,x_{{4}},\\
&\psi_{AK_2}(x_5)=x_{{1}}+30\,x_{{2}}+150\,x_{{3}}+240\,x_{{4}}+120\,x_{{5}},\\
&\psi_{AK_2}(x_6)=x_{{1}}+62\,x_{{2}}+540\,x_{{3}}+1560\,x_{{4}}+1800\,x_{{5}}+720\,x_{{
6}}
.
\end{align*}
Let us prove the following statement
\begin{lemma}
$$
B(n,k)= k!\, S(n,k).
$$
\end{lemma}
\begin{proof}
We have $B(0,0)=1,$ $B(1,1)=0,B(2,1)=1,B(2,2)=2,$ thus  the initial conditions hold.
By using the summation formula
$$
\sum_{i=1}^n a_i \sum_{j=0}^{i-1} b_{i,j} c_j=\sum_{i=0}^{n-1} c_i \sum_{j=i+1}^{n} a_j b_{j,i},
$$
we have 
\begin{gather*}
D_{\mathcal{K}_2}\left(\psi_{AK_2}(x_n)\right)=D_{\mathcal{K}_2}\left(\sum_{i=1}^n B(n,i)x_i\right)=\sum_{i=1}^n B(n,i)\sum_{i=0}^{i-1} \frac{(-1)^{i+1-j}}{i-j}x_j=\\=
\sum_{i=0}^{n-1} x_i \sum_{j=i+1}^{n}  \frac{(-1)^{j+1-i}}{j-i} B(n,j).
\end{gather*}
On the other hand
$$
D_{\mathcal{K}_2}(\psi_{AK_2}(x_n))=\psi_{AK_2}(\mathcal{D}(x_n))=n \psi_{AK_2}(x_{n-1})=n \sum_{i=1}^{n-1} B(n-1,i) x_{i}.
$$
By equation the corresponding coefficients we obtain that the numbers $B(n,k)$ satisfy  the following system of  recurrence relations:
$$
\displaystyle \sum_{j=i+1}^n \frac{(-1)^{j+1-i}}{j-i} B(n,j)= n B(n-1,i),i=0,1,\ldots, n-1.\\
$$

To prove that the numbers $B(n,k)=k! S(n,k)$ are the solutions of the system  we use the Lemma 3.1. Under conditions of this lemma we put 
$$
a_i=\frac{(-1)^{i+1}}{i}, A(n,i)=B(n,i).
$$
Then 
$$
g(x)=\sum_{i=1}^{\infty} a_i z^n =\sum_{i=1}^{\infty} \frac{(-1)^{i+1}}{i} z^n=\ln(1+z).
$$
Also use the exponential generating function for the Stirling numbers of the second kind 
$$
\sum_{n=1}^{\infty} j! S(n,j) \frac{x^n}{n!}=(e^{x}-1)^j.
$$
It follows that  $f(z)=e^{x}-1.$  Since  $g(f(z))=z,$ then all conditions of the Lemma 3.1 hold. Thus the numbers  $B(n,i)$ are the solutions of the system of recurrence equations:
$$
\displaystyle \sum_{j=i+1}^n \frac{(-1)^{j+1-i}}{j-i} B(n,j)= n B(n-1,i),i=0,1,\ldots, n-1.
$$\end{proof}


\subsection{Some identities}


  Having  the explicit forms of the intertwining maps we may assist to any series of elements of kernels of the Weitzenb\"ock derivations $\mathcal{D}$ an   identity for the Kravchuk polynomials. A  list of such  series are presented in the paper  \cite{B_App}. For instance, the polynomials 
$$
I_n=\frac{1}{2}\sum_{i=0}^{2n} (-1)^i {2n \choose i} x_i x_{n-i},
$$
belong to the kernel of the derivation $\mathcal{D}$, thus $\psi_{AK_1}(I_n)=\varphi_1(a), \psi_{AK_2}(I_n)=\varphi_2(a),$ where   $\varphi_1,$ and $ \varphi_2$ are   polynomials of one variable. It is a difficult problem to find the explicit form of the polynomials $\varphi_1,$ and $ \varphi_2$.

Another  example of the element of the kernel  $\mathcal{D}$ is the determinant of the Hankel matrix
$$
H_n:=\det(\psi_{AK_1}(x_{i+j-2}))=
\begin{vmatrix} x_0 &  x_1 & x_2 &\cdots  & x_n  \\
x_1 &  x_2 & x_3 &\cdots  & x_{n+1}  \\
\hdotsfor{5}\\
 x_{n-1} &  x_{n} & x_{n-1} &\cdots  & x_{2n-1}\\
 x_n &  x_{n+1} & x_{n+2} &\cdots  & x_{2n}  
\end{vmatrix}.
$$ 
For  $\psi_{AK_1}(H_n),\psi_{AK_2}(H_n)$ we offer the conjecture:

\noindent
\textbf{Conjecture 3.}
$$
\begin{array}{ll}
(i) &\displaystyle \psi_{AK_1}(H_n)=(-1)^{\frac{n(n+1)}{2}} \prod_{i=0}^{n-1} i!  \prod_{i=0}^{n-2} (i+a)^{n-1-i},\\
(ii) & \displaystyle  \psi_{AK_2}(H_n)=(-1)^{\frac{n(n+1)}{2}} \prod_{i=0}^n 2^i i!  \prod_{i=0}^{n-2} (x-i)^{n-1-i}.
\end{array}
$$




\end{document}